\documentclass[a4paper, 12pt,reqno]{amsart}% Будь ласка, використовуйте клас документа amsart

\usepackage[T2A]{fontenc}
\usepackage[cp1251]{inputenc}% Змініть кодування, якщо потрібно
\usepackage[ukrainian]{babel}
\usepackage{amsmath}
\usepackage{tikz}

\usepackage{amsmath}
 \usepackage{graphicx}% Якщо потрібно додавати ілюстрації в документ
\usepackage{textcase}% Корисно для української чи російської мови
\usepackage{amssymb}
\theoremstyle{plain}
\newtheorem{theorem}{Теорема}
\newtheorem*{theorem*}{Теорема}

\newtheorem*{corollary*}{Наслідок}
\newtheorem{lemma}{Лема}
\newtheorem*{lemma*}{Лема}

\newtheorem*{proposition*}{Твердження}

\newtheorem*{conjecture*}{Гіпотеза}
\theoremstyle{definition}
\newtheorem{definition}{Означення}
\newtheorem*{definition*}{Означення}
\theoremstyle{remark}
\newtheorem{remark}{Зауваження}
\newtheorem*{remark*}{Зауваження}
\begin{document}
2010 Mathematics Subject Classification 11K55, 26A09

УДК 517.5+517.13+511.72
\title[Нега-$\tilde Q$-представлення]{Нега-$\tilde Q$-зображення   дійсних чисел}
\author{ С. О. Сербенюк  }
\address{Інститут математики НАН України\\
         Київ, Україна}

\email{simon6@ukr.net; ~ simon.mathscience@imath.kiev.ua}

\begin{abstract}
У даній статті побудовано  нега-$\tilde Q$-представлення дійсних чисел, яке є узагальненням представлення чисел знакопочережними рядами Кантора. Для моделювання нега-$\tilde Q$-представлення використовується аналітичний та геометричний підходи. Недоліки та переваги кожного із способів досліджено, шукане представлення змодельовано. 

\textsc{Abstract.} The article is devoted to  modeling of the nega-$\tilde Q$-representation of real numbers. The representation is generalization of representation by  alternating Cantor series. Analytic and geometric approach are used for modeling of nega-$\tilde Q$-representation.  
Advantages and disadvantages of these approaches are investigated, the representation is  modeled.
\end{abstract}

\maketitle

\section{Вступ}

Нехай заданою є матриця $\tilde Q=||q_{i,j}||$, де  $i=\overline{0,m_j}$, $m_j \in N^{0} _{\infty}= \mathbb N \cup \{0,\infty\}$, $j=1,2,...$, для якої справедливою є наступна система властивостей:
\begin{equation}
\left\{
\begin{aligned}
\label{eq: tilde Q 1}
1^{\circ}. ~~~~~~~~~~~~~~~~~~~~~~~~~~ ~~~ ~~ q_{i,j}>0;\\
2^{\circ}.  ~~~~~~~~~~~~~~~\forall j \in \mathbb N: \sum^{m_j}_{i=0} {q_{i,j}}=1;\\
3^{\circ}.  \forall (i_j), i_j \in \mathbb N \cup \{0\}: \prod^{\infty} _{j=1} {q_{i_j,j}}=0.\\
\end{aligned}
\right.
\end{equation}
\begin{lemma}[{\cite[с.~88]{Pra98}}]
Для будь-якого $x \in [0;1)$ існує послідовність $(i_{k}(x))$, $i_{k}(x)\in~N ^{0} _{m_{k}}\equiv\\ \equiv~\{0,1,...,m_k\}$, така, що
\begin{equation}
\label{def: tilde Q 1}
x=a_{i_1(x),1}+\sum^{\infty} _{k=2} {\left[a_{i_k(x),k}\prod^{k-1} _{j=1} {q_{i_j(x),j}} \right]},
\end{equation}
де
$$
a_{i_k,k}=\begin{cases}
\sum^{i_k-1} _{i=0} {q_{i,k}},&\text{якщо $i_k \ne 0$,}\\
0,&\text{якщо $i_k=0$.}
\end{cases}
$$
\end{lemma}

\begin{definition}
Подання числа $x \in [0;1)$ у виглядi розкладу \eqref{def: tilde Q 1} називають~{\cite[с.~89]{Pra98}}  \emph{$\tilde Q$ -представленням (або $\tilde Q$-розкладом) числа $x \in [0;1)$} i позначають $x=\Delta^{\tilde Q} _{i_1(x)i_2(x)...i_{k}(x)...}$. Останнiй запис називають \emph{$\tilde Q$-зображенням числа $x$}.
\end{definition}

$\tilde Q$-представлення дійсних чисел, очевидно, є узагальненням  представлення дійсних чисел знакододатними рядами Кантора \cite{ Symon3, Ralko10, Cantor1}
$$
\frac{\varepsilon_{1}}{d_1}+\frac{\varepsilon_{2}}{d_1d_2}+...+\frac{\varepsilon_{n}}{d_1d_2...d_n}+...,
$$
де $(d_n)$ --- фіксована послідовність натуральних чисел $d_n$, більших $1$, $(A_{d_n})$ --- послідовність алфавітів $A_{d_n} \equiv \{0,1,...,d_n-1\}$, $\varepsilon_{n}\in A_{d_n}$.
Останні ряди, в свою чергу, є узагальненням класичного s-го представлення
$$
\sum^{\infty} _{n=1} {\frac{\alpha_{n}}{s^{n}}}, ~\alpha_n \in A\equiv \{0,1,...,s-1\},
$$
де $1<s$ --- фіксоване натуральне число.

З метою побудови нових об'єктів теорій неперервних ніде недиференційовних функцій,  сингулярних функцій, теорії ймовірностей та фрактального аналізу змоделюємо представлення,  яке є узагальненням  нега-s-го представлення 
$$
\sum^{\infty} _{n=1}{\frac{\alpha_n}{(-s)^n}}\equiv\Delta^{-s} _{\alpha_1\alpha_2...\alpha_n...},~~~\alpha_n\in A,
$$
 та представлення дійсних чисел знакопочережним рядом Кантора (нега-D-представлення)~\cite{Symon-altern.series-2015, Symon4}
$$
\sum^{\infty} _{n=1}{\frac{(-1)^n\varepsilon_n}{d_1d_2...d_n}}\equiv\Delta^{-D} _{\varepsilon_1\varepsilon_2...\varepsilon_n...}, ~\varepsilon_n\in A_{d_n},
$$
 та дослідимо основні властивості як системи числення побудованого представлення.

Отож, нехай маємо слідуючі розклади дійсних чисел
\begin{equation}
\label{def: nega-tilde Q 1}
-a_{i_1,1}+\sum^{\infty} _{n=2}{\left[(-1)^na_{i_n,n}\prod^{n-1} _{j=1} {q_{i_j,j}}\right]}
\end{equation}
та
\begin{equation}
\label{def: nega-tilde Q 2}
\sum^{i_1-1} _{i=0} {q_{i,1}}+\sum^{\infty} _{n=2}{\left[(-1)^{n-1}\tilde \delta_{i_n,n}\prod^{n-1} _{j=1}{\tilde q_{i_j,j}}\right]}+\sum^{\infty} _{n=1}{\left(\prod^{2n-1} _{j=1}{\tilde q_{i_j,j}}\right)}, 
\end{equation}
де
$$
\tilde q_{i_n,n}=\begin{cases}
q_{i_n,n},&\text{якщо $n$  --- непарне;}\\
q_{m_n-i_n,n},&\text{якщо $n$ --- парне,}
\end{cases}
$$
$$
\tilde \delta_{i_n,n}=\begin{cases}
0,&\text{якщо $i_n=0$ та $n$ --- непарне;}\\
\sum^{i_n-1} _{i=0} {q_{i,n}},&\text{якщо $i_{n}\ne 0$ та $n$ --- непарне;}\\
\sum^{m_{n}} _{i=m_{n}-i_{n}} {q_{i,n}},&\text{якщо  $n$ --- парне.}
\end{cases}
$$

\section{Дослідження  ряду \eqref{def: nega-tilde Q 1} як системи числення}

Ряд \eqref{def: nega-tilde Q 1} можна побудувати, використовуючи аналітичний підхід до побудови системи числення. 
Суть аналітичного підходу до побудови  знакопочережного $\tilde Q$-представлення дійсних чисел полягає в тому, що якщо  основою нега-s-го представлення є фіксоване число $(-s)$, де $1<s\in \mathbb N$, та за основу знакопочережного канторівського представлення приймається фіксована послідовність $(-d_n)$, де $1<d_n\in \mathbb N$, то знакопочережне $\tilde Q$-представлення можна побудувати, прийнявши за основу системи числення (представлення дійсних чисел)  фіксовану матрицю $(-1)\cdot \tilde Q=(-1)\cdot ||q_{i,j}||$, $i=\overline{0,m_j}, m_j \in \mathbb N \cup \{0,\infty\}$, j=1,2,..., для елементів $|q_{i_j,j}|$ якої справедливою є система умов \eqref{eq: tilde Q 1}. Отже, розглянемо суму
$$
\sum^{\infty} _{n=1}{\left[\left(\sum^{i_n-1} _{i=0}{(-q_{i,n})}\right)\prod^{n-1} _{j=1}{(-q_{i_j,j})}\right]}=-a_{i_1,1}+\sum^{\infty} _{n=2}{\left[(-1)^na_{i_n,n}\prod^{n-1} _{j=1}{q_{i_j,j}}\right]}.
$$

Очевидно, що останній знакопочережний ряд є абсолютно збіжним, причому його сума належить відрізку $[t^{'} _0;t^{''} _0]$, де
$$
t^{'} _0=-a_{m_1,1}+\sum^{\infty} _{n=2}{\left[(-1)^n\tilde a_{m_n,n}\prod^{n-1} _{j=1}{\tilde q_{m_j,j}}\right]}=-1+\sum^{\infty} _{n=1}{\left((-1)^{n-1}\prod^{n} _{j=1}{\tilde q_{m_j,j}}\right)},
$$
$$
t^{''} _0=\sum^{\infty} _{n=2}{\left[(-1)^n\tilde a_{0,n}\prod^{n-1} _{j=1}{\tilde q_{0,j}}\right]}=\sum^{\infty} _{n=1}{\left((-1)^{n-1}\prod^{n} _{j=1}{\tilde q_{0,j}}\right)},
$$
$$
\tilde a_{i_n,n}=\begin{cases}
a_{i_n,n},&\text{якщо $n$  --- непарне;}\\
a_{m_n-i_n,n},&\text{якщо $n$ --- парне.}
\end{cases}
$$

$$
t^{''} _0-t^{'} _0=a_{m_1,1}+a_{m_2,2}q_{0,1}+a_{m_3,3}q_{m_1,1}q_{0,2}+a_{m_4,4}q_{0,1}q_{m_2,2}q_{0,3}+a_{m_5,5}q_{m_1,1}q_{0,2}q_{m_3,3}q_{0,4}+....
$$

Факт представлення числа $x$ у вигляді розкладу  \eqref{def: nega-tilde Q 1}  позначатимемо $\Delta^{(-\tilde Q)} _{i_1i_2...i_n...}$. 

Очевидно, при $q_{i,j}=\frac{1}{s}$ для всіх $j \in \mathbb N$,  $i=\overline{0,s-1}$, ряд \eqref{def: nega-tilde Q 1} набуває виду нега-s-го представлення
$$
\left[-\frac{s}{s+1};\frac{1}{s+1}\right]\ni x =\Delta^{-s} _{\alpha_1\alpha_2...\alpha_n...}\equiv\sum^{\infty} _{n=1} {\frac{(-1)^{n}\alpha_n}{s^n}},
 $$
а у випадку  $q_{i,j}=\frac{1}{d_j}$ для будь-якого  $j \in \mathbb N$, $i=\overline{0,d_j-1}$, --- знакопочережного ряду Кантора (нега-D-представлення)
$$
\left[-\sum^{\infty}_{k=1}{\frac{d_{2k-1}-1}{d_1d_2...d_{2k-1}}};\sum^{\infty}_{k=1}{\frac{d_{2k}-1}{d_1d_2...d_{2k}}}\right]\ni x=\Delta^{-D} _{\varepsilon_1\varepsilon_2...\varepsilon_n...}\equiv \sum^{\infty} _{n=1}{\frac{(-1)^n\varepsilon_n}{d_1d_2...d_n}}.
$$

Введемо допоміжне для подальшого дослідження представлення дійсних чисел рядом \eqref{def: nega-tilde Q 1} поняття циліндричної множини.
\begin{definition} 
\emph{Циліндричною множиною (або циліндром) $\Delta^{(-\tilde Q)} _{c_1c_2...c_n}$ рангу $n$ з основою $c_1c_2...c_n$} називається множина виду
$$
\Delta^{(-\tilde Q)} _{c_1c_2...c_n}\equiv \left\{x: x=\Delta^{(-\tilde Q)} _{c_1c_2...c_ni_{n+1}i_{n+2}...i_{n+k}...}, x\in[t^{'} _0;t^{''} _0] \right\},
$$
де $c_1,...,c_n$ --- фіксовані числа, $i_{n+k} \in N^0 _{m_{n+k}}, k=1,2,...$.
\end{definition}

Дослідимо, чи можна вважати ряд \eqref{def: nega-tilde Q 1}  системою числення та проведемо короткий порівняльний аналіз даного розкладу дійсних чисел і згаданих вище знакопочережних представлень. 

\begin{enumerate}
\item \emph{Основне метричне відношення}.
В нега-s-му, нега-канторівському представленнях для відповідних циліндрів $\Delta^{-s} _{c_1c_2...c_n}$, $\Delta^{-D} _{c_1c_2...c_n}$  рангу $n$ з основою $c_1c_2...c_n$ справедливими є наступні співвідношення:
\begin{itemize}
\item для довільного $n \in \mathbb N$
$$
|\Delta^{-s} _{c_1c_2...c_n}|=\frac{1}{s^n}\hat \varphi^n (\Delta^s _{(s-1)})=\frac{1}{s^n}\hat \varphi^n\left(\left|\sup_x{\Delta^{-s} _{\alpha_1\alpha_2...\alpha_n...}}-\inf_x{\Delta^{-s} _{\alpha_1\alpha_2...\alpha_n...}}\right|\right)=
$$
$$
=\frac{1}{s^n}\left|\hat \varphi^n\left(\sup_x{\Delta^{-s} _{\alpha_1\alpha_2...\alpha_n...}}\right)-\hat \varphi^n\left(\inf_x{\Delta^{-s} _{\alpha_1\alpha_2...\alpha_n...}}\right)\right|=\frac{1}{s^n}\left|\hat \varphi^n\left({\Delta^{-s} _{(0[s-1])}}\right)-\hat \varphi^n\left({\Delta^{-s} _{([s-1]0)}}\right)\right|,
$$
де $\hat \varphi$ --- оператор зсуву цифр.

Таким чином, 
$$
\left|\hat \varphi^n\left(\sup_x{\Delta^{-s} _{\alpha_1\alpha_2...\alpha_n...}}\right)-\hat \varphi^n\left(\inf_x{\Delta^{-s} _{\alpha_1\alpha_2...\alpha_n...}}\right)\right|=\hat \varphi^n\left(\sup_x{\Delta^{-s} _{\alpha_1\alpha_2...\alpha_n...}}-\inf_x{\Delta^{-s} _{\alpha_1\alpha_2...\alpha_n...}}\right)=const=
$$
$$
=\sup_x{\Delta^{-s} _{\alpha_1\alpha_2...\alpha_n...}}-\inf_x{\Delta^{-s} _{\alpha_1\alpha_2...\alpha_n...}}.
$$

\item для кожного $n \in \mathbb N$
$$
|\Delta^{-D} _{c_1c_2...c_n}|=\frac{1}{d_1d_2...d_n}\hat \varphi^n (\Delta^D _{[d_1-1][d_2-1]...[d_n-1]...})=\frac{1}{d_1d_2...d_n}\hat \varphi^n\left(\left|\sup_x{\Delta^{-D} _{\varepsilon_1\varepsilon_2...\varepsilon_n...}}-\inf_x{\Delta^{-D} _{\varepsilon_1\varepsilon_2...\varepsilon_n...}}\right|\right)=
$$
$$
=\frac{1}{d_1d_2...d_n}\left|\hat \varphi^n\left(\sup_x{\Delta^{-D} _{\varepsilon_1\varepsilon_2...\varepsilon_n...}}\right)-\hat \varphi^n\left(\inf_x{\Delta^{-D} _{\varepsilon_1\varepsilon_2...\varepsilon_n...}}\right)\right|=
$$
$$
=\frac{1}{d_1d_2...d_n}\left|\hat \varphi^n\left({\Delta^{-D} _{0[d_2-1]0[d_4-1]...0[d_{2k}-1]...}}\right)-\hat \varphi^n\left({\Delta^{-D} _{[d_1-1]0[d_3-1]0...[d_{2k-1}-1]0...}}\right)\right|,
$$
де $k=1,2,...$, $\hat \varphi$ --- оператор зсуву цифр відповідного представлення.

Отже, 
$$
\hat \varphi^n\left(\sup_x{\Delta^{-D} _{\varepsilon_1\varepsilon_2...\varepsilon_n...}}-\inf_x{\Delta^{-D} _{\varepsilon_1\varepsilon_2...\varepsilon_n...}}\right)=\left|\hat \varphi^n\left(\sup_x{\Delta^{-D} _{\varepsilon_1\varepsilon_2...\varepsilon_n...}}\right)-\hat \varphi^n\left(\inf_x{\Delta^{-D} _{\varepsilon_1\varepsilon_2...\varepsilon_n...}}\right)\right|=const=
$$
$$
=\sup_x{\Delta^{-D} _{\varepsilon_1\varepsilon_2...\varepsilon_n...}}-\inf_x{\Delta^{-D} _{\varepsilon_1\varepsilon_2...\varepsilon_n...}}.
$$
\end{itemize}

Розглянемо розклад \eqref{def: nega-tilde Q 1}. 

Позначатимемо символом $d(\cdot)$ діаметр множини. Розглянемо випадок, коли $n=1$. В такому разі
$$
d\left(\Delta^{(-\tilde Q)} _{c_1}\right)=\Delta^{(-\tilde Q)} _{c_1m_20m_4...0m_{2k}...}-\Delta^{(-\tilde Q)} _{c_10m_30m_5...0m_{2k-1}...}=
$$
$$
=q_{c_1,1}\left(a_{m_2,2}+\sum^{\infty} _{k=3}{\left[\tilde a_{0,k}\prod^{k-1} _{j=2}{\tilde q_{0,j}}\right]}+\sum^{\infty} _{k=3}{\left[\tilde a_{m_{k},k}\prod^{k-1} _{j=2}{\tilde q_{m_j,j}}\right]}\right)=
$$
$$
= q_{c_1,1}\left(\frac{t^{''} _0}{q_{0,1}}-\frac{t^{'} _0+a_{m_1,1}}{q_{m_1,1}}\right)=\frac{q_{m_1,1}t^{''} _0-q_{0,1}t^{'} _0-a_{m_1,1}q_{0,1}}{q_{0,1}q_{m_1,1}}q_{c_1,1}=q_{c_1,1}(t^{''} _0-t^{'} _0),
$$
якщо справедливою є умова
$$
t^{''} _0=\frac{q_{0,1}(1-q_{m_1,1})}{q_{m_1,1}(1-q_{0,1})}(1+t^{'} _0).
$$

Очевидно, остання умова справджується не у всіх випадках залежно від матриці $\tilde Q$. Наприклад, розглянемо випадок, коли для всіх натуральних значень $n$ справедливо $m_n<\infty$ та для всіх $n>1$ $q_{0,n}=const=q_0$, $q_{m_n}=const=q_m$, $q_0\ne q_m$,
$q_{0,1}=q_1$, $q_{m_1,1}=q_2$, $q_1\ne q_2$. Тоді
$$
t^{'} _0=\Delta^{(-\tilde Q)} _{m_10m_30m_5...}\equiv-\left(1-q_2+q_2\sum^{\infty} _{k=1}{\left((1-q_m)q^k _0 q^{k-1} _m\right)}\right)=q_2\left(1-\frac{(1-q_m)q_0}{1-q_0q_m}\right)-1,
$$

$$
t^{''} _0=\Delta^{(-\tilde Q)} _{0m_20m_40m_6...}\equiv q_1\left(\sum^{\infty} _{k=1}{(1-q_m)q^{k-1} _0q^{k-1} _m}\right)=\frac{1-q_m}{1-q_0q_m}q_1.
$$

Отож, 
$$
\frac{q_{0,1}(1-q_{m_1,1})}{q_{m_1,1}(1-q_{0,1})}(1+t^{'} _0)=\frac{q_1(1-q_2)}{q_2(1-q_1)}q_2\left(1-\frac{(1-q_m)q_0}{1-q_0q_m}\right)=\frac{q_1(1-q_2)(1-q_0)}{(1-q_1)(1-q_0q_m)}=t^{''} _0,
$$
якщо
$$
\frac{1-q_m}{1-q_0}=\frac{1-q_2}{1-q_1}.
$$
Цілком очевидно, що останнє співвідношення  справедливе не завжди.

Тобто, метричне відношення 
$$
\frac{d\left(\Delta^{(-\tilde Q)} _{c_1c_2...c_kc}\right)}{d\left(\Delta^{(-\tilde Q)} _{c_1c_2...c_k}\right)},
$$
де $k=0,1,2,3,...,$ при $k=0$   циліндричною множиною вважатимемо відрізок $[t^{'} _0;t^{''} _0]$, залежно від матриці $\tilde Q$ не завжди дорівнює $q_{c,k+1}$. В загальному випадку
$$
\frac{d\left(\Delta^{(-\tilde Q)} _{c_1c_2...c_kc}\right)}{d\left(\Delta^{(-\tilde Q)} _{c_1c_2...c_k}\right)}=q_{c,k+1}\frac{a_{m_{k+2},k+2}+\sum^{\infty} _{t=k+3}{\left[\tilde a_{m_{t},t}\prod^{t-1} _{r=k+2}{\tilde q_{m_r,r}}\right]}+\sum^{\infty} _{t=k+3}{\left[\tilde a_{0,t}\prod^{t-1} _{r=k+2}{\tilde q_{0,r}}\right]}}{a_{m_{k+1},k+1}+\sum^{\infty} _{t=k+2}{\left[\tilde a_{m_{t},t}\prod^{t-1} _{r=k+1}{\tilde q_{m_r,r}}\right]}+\sum^{\infty} _{t=k+2}{\left[\tilde a_{0,t}\prod^{t-1} _{r=k+1}{\tilde q_{0,r}}\right]}}.
$$

\item \emph{Представлення чисел, що мають два різних зображення}. Виявляється \cite{Symon-altern.series-2015}, лише числа із зліченної множини  можуть мати два різних  нега-D-зображення. Аналогічна властивість справедлива і для нега-s-го зображення. Зокрема,
$$
\Delta^{-s} _{\alpha_1\alpha_2...\alpha_{n-1}\alpha_n([s-1]0)}=\Delta^{-s} _{\alpha_1\alpha_2...\alpha_{n-1}[\alpha_n-1](0[s-1])},~\alpha_n\ne0,
 $$ 
 $$
\Delta^{-D} _{\varepsilon_1\varepsilon_2...\varepsilon_{n-1}\varepsilon_n[d_{n+1}-1]0[d_{n+3}-1]0[d_{n+5}-1]...}=\Delta^{-D} _{\varepsilon_1\varepsilon_2...\varepsilon_{n-1}[\varepsilon_n-1]0[d_{n+2}-1]0[d_{n+4}-1]0[d_{n+6}-1]...},~\varepsilon_n\ne0.
$$
 Для ряду  \eqref{def: nega-tilde Q 1} справедливим є наступне твердження.
 \begin{lemma} Якщо
 $$
 \Delta^{(-\tilde Q)} _{c_1c_2...c_{n-1}c_nm_{n+1}0m_{n+3}0m_{n+5}...}=\Delta^{(-\tilde Q)} _{c_1c_2...c_{n-1}[c_n-1]0m_{n+2}0m_{n+4}0m_{n+6}...},~c_n\ne0,
 $$
 тоді  при парному $n$  справедливою є наступна рівність
$$
\frac{t^{''} _0-\sum^{n-2} _{k=2}{\left[\tilde a_{0,k}\prod^{k-1} _{j=1}{\tilde q_{0,j}}\right]}-\prod^{n-1} _{j=1}{\tilde q_{0,j}}}{t^{'} _0+a_{m_1,1}+\sum^{n-1} _{k=2}{\left[\tilde a_{m_k,k}\prod^{k-1} _{j=1}{\tilde q_{m_j,j}}\right]}}=\frac{q_{c_n,n}\prod^{n} _{j=1}{\tilde q_{0,j}}}{q_{c_n-1,n}\prod^{n} _{j=1}{\tilde q_{m_j,j}}}
$$
  та 
$$
\frac{t^{''} _0-\sum^{n-1} _{k=2}{\left[\tilde a_{0,k}\prod^{k-1} _{j=1}{\tilde q_{0,j}}\right]}}{t^{'} _0+a_{m_1,1}+\sum^{n-2} _{k=2}{\left[\tilde a_{m_k,k}\prod^{k-1} _{j=1}{\tilde q_{m_j,j}}\right]}+\prod^{n-1} _{j=1}{\tilde q_{m_j,j}}}=\frac{q_{c_n-1,n}\prod^{n} _{j=1}{\tilde q_{0,j}}}{q_{c_n,n}\prod^{n} _{j=1}{\tilde q_{m_j,j}}}
$$
 при непарному $n$.
 \end{lemma}
 \begin{proof} Нехай  
 $$
 x_1=\Delta^{(-\tilde Q)} _{c_1c_2...c_{n-1}c_nm_{n+1}0m_{n+3}0m_{n+5}...},~ \Delta^{(-\tilde Q)} _{c_1c_2...c_{n-1}[c_n-1]0m_{n+2}0m_{n+4}0m_{n+6}...}=x_2,~c_n\ne0.
 $$
 Якщо $n$ --- парне  число, тоді
 $$
 x_1=-a_{c_1,1}+\sum^{n-1} _{k=2}{\left[(-1)^ka_{c_k,k}\prod^{k-1} _{j=1}{q_{c_j,j}}\right]}+(-1)^n\left(\prod^{n-1} _{j=1}{q_{c_j,j}}\right)\times
 $$
 $$
 \times\left(a_{c_n,n}-q_{c_n,n}\left(a_{m_{n+1},n+1}+\sum^{\infty} _{k=2}{\left[ a_{m_{n+2k-1},n+2k-1}\prod^{n+2k-2} _{j=n+1}{\tilde q_{m_{j},j}}\right]}\right)\right)=
 $$
 $$
 =-a_{c_1,1}+\sum^{n-1} _{k=2}{\left[(-1)^ka_{c_k,k}\prod^{k-1} _{j=1}{q_{c_j,j}}\right]}+\left(\prod^{n-1} _{j=1}{q_{c_j,j}}\right)\times
 $$
 $$
 \times \left(a_{c_n,n}+q_{c_n,n}\left({t^{'} _0+a_{m_1,1}+\sum^{n-1} _{k=2}{\left[\tilde a_{m_{k},k}\prod^{k-1} _{j=n+1}{\tilde q_{m_{j},j}}\right]}}\right)\prod^{n} _{j=1}{(\tilde q_{m_{j},j})^{-1}}\right). 
 $$
 $$
 x_2=-a_{c_1,1}+\sum^{n-1} _{k=2}{\left[(-1)^ka_{c_k,k}\prod^{k-1} _{j=1}{q_{c_j,j}}\right]}+(-1)^n\left(\prod^{n-1} _{j=1}{q_{c_j,j}}\right)\times
 $$
 $$
 \times\left(a_{c_n-1,n}+q_{c_n-1,n}\sum^{\infty} _{k=1}{\left[ a_{m_{n+2k},n+2k}\prod^{n+2k-1} _{j=n+1}{\tilde q_{0,j}}\right]}\right)=
$$
 $$
 =-a_{c_1,1}+\sum^{n-1} _{k=2}{\left[(-1)^ka_{c_k,k}\prod^{k-1} _{j=1}{q_{c_j,j}}\right]}+\left(\prod^{n-1} _{j=1}{q_{c_j,j}}\right)\times
 $$
 $$
 \times\left(a_{c_n-1,n}+q_{c_n-1,n}\left(t^{''} _0-\sum^{n} _{k=2}{\left[\tilde a_{0,k}\prod^{k-1} _{j=n+1}{\tilde q_{0,j}}\right]}\right)\prod^{n} _{j=1}{(\tilde q_{0,j})^{-1}}\right). 
 $$
 З умови $x_1=x_2$  й слідує перша рівність в умові леми. 
 
 Нехай $n$ --- непарне  число. Тоді при умові, що  $x_1=x_2$ отримаємо
 
$$
-q_{c_n-1,n}+q_{c_n,n}\left(t^{''} _0-\sum^{n-1} _{k=2}{\left[\tilde a_{0,k}\prod^{k-1} _{j=1}{\tilde q_{0,j}}\right]}\right)\prod^{n} _{j=1}{(\tilde q_{0,j})^{-1}}=
$$
$$
=q_{c_n-1,n}\left(t^{'} _0+a_{m_{1},1}+\sum^{n} _{k=2}{\left[\tilde a_{m_k,k}\prod^{k-1} _{j=1}{\tilde q_{m_j,j}}\right]}\right)\prod^{n} _{j=1}{(\tilde q_{m_j,j})^{-1}},
$$
 звідки й слідує друга рівність в умові леми.
  \end{proof}
 
 Додатково розглянемо  ще деякі властивості циліндрів $\Delta^{(-\tilde Q)} _{c_1c_2...c_n}$ рангу $n$ з основою $c_1c_2...c_n$. 
 
 \begin{lemma}
 Циліндр $\Delta^{(-\tilde Q)} _{c_1c_2...c_n}$ є відрізком.
 \end{lemma}
 \begin{proof}
 Доведення проведемо для парного $n$. Нехай $x\in \Delta^{(-\tilde Q)} _{c_1c_2...c_n}$. Тобто, 
 $$
 x=-a_{c_1,1}+\sum^{n} _{k=2}{\left[(-1)^ka_{c_k,k}\prod^{k-1} _{j=1}{q_{c_j,j}}\right]}+\left(\prod^{n} _{j=1}{q_{c_j,j}}\right)\left(-a_{i_{n+1},n+1}+\sum^{\infty} _{l=n+2}{\left[(-1)^la_{i_l,l}\prod^{l-1} _{r=n+1}{q_{i_r,r}}\right]}\right).
 $$
 Звідси,
 $$
 x^{'}=-a_{c_1,1}+\sum^{n} _{k=2}{\left[(-1)^ka_{c_k,k}\prod^{k-1} _{j=1}{q_{c_j,j}}\right]}-\left(\prod^{n} _{j=1}{q_{c_j,j}}\right)\left(a_{m_{n+1},n+1}+\sum^{\infty} _{t=2}{\left[\tilde a_{m_{n+t},n+t}\prod^{n+t-1} _{r=n+1}{\tilde q_{m_r,r}}\right]}\right)\le
 $$
 $$
 \le x\le -a_{c_1,1}+\sum^{n} _{k=2}{\left[(-1)^ka_{c_k,k}\prod^{k-1} _{j=1}{q_{c_j,j}}\right]}+\left(\prod^{n} _{j=1}{q_{c_j,j}}\right)\sum^{\infty} _{t=2}{\left[\tilde a_{0,n+t}\prod^{n+t-1} _{r=n+1}{\tilde q_{0,r}}\right]}=x^{''}. 
 $$
 
 Отже, $x\in [x^{'};x^{''}]\supseteq \Delta^{(-\tilde Q)} _{c_1c_2...c_n}$.
 
 Оскільки
 $$
 x^{'}=-a_{c_1,1}+\sum^{n} _{k=2}{\left[(-1)^ka_{c_k,k}\prod^{k-1} _{j=1}{q_{c_j,j}}\right]}+\left(\prod^{n} _{j=1}{q_{c_j,j}}\right)\inf{\left\{-a_{i_{n+1},n+1}+\sum^{\infty} _{l=n+2}{\left[(-1)^la_{i_l,l}\prod^{l-1} _{r=n+1}{q_{i_r,r}}\right]}\right\}},
 $$
 $$
 x^{''}=-a_{c_1,1}+\sum^{n} _{k=2}{\left[(-1)^ka_{c_k,k}\prod^{k-1} _{j=1}{q_{c_j,j}}\right]}+\left(\prod^{n} _{j=1}{q_{c_j,j}}\right)\sup{\left\{-a_{i_{n+1},n+1}+\sum^{\infty} _{l=n+2}{\left[(-1)^la_{i_l,l}\prod^{l-1} _{r=n+1}{q_{i_r,r}}\right]}\right\}},
 $$
 то $x^{'},x^{''},x$ належать  $\Delta^{(-\tilde Q)} _{c_1c_2...c_n}$.
 \end{proof}
 
 \item \emph{Розташування циліндрів однакового рангу.} Оскільки у згаданих вище нега-s-му та нега-D-представленні циліндричні множини є відрізками, що розташовані ''зліва направо'' при парному $n$ та~''~справа наліво'',~якщо $n$ --- непарне, то ''щось подібне'' мало б  справджуватися і для рядів \eqref{def: nega-tilde Q 1}.
 
 Розглянемо необхідні для подальшого дослідження розташування циліндрів $\Delta^{(-\tilde Q)} _{c_1c_2...c_n}$  співвідношення. 
 
 Нехай $n$ --- деяке фіксоване натуральне число. Якщо циліндри $\Delta^{(-\tilde Q)} _{c_1c_2...c_{n-1}c}$, $\Delta^{(-\tilde Q)} _{c_1c_2...c_{n-1}[c+1]}$ \emph{перекриваються} і розташовані: 
 \begin{itemize}
 \item \emph{ ''зліва направо''}, то
 $$
\kappa_1=\sup {\Delta^{(-\tilde Q)} _{c_1c_2...c_{n-1}c}}-\inf {\Delta^{(-\tilde Q)} _{c_1c_2...c_{n-1}[c+1]}}> 0;
 $$
 \item \emph{ ''справа наліво''}, то
 $$
\kappa_2=\sup{\Delta^{(-\tilde Q)} _{c_1c_2...c_{n-1}[c+1]}} -\inf {\Delta^{(-\tilde Q)} _{c_1c_2...c_{n-1}c}}> 0.
 $$
 \end{itemize}
 
 Причому, у першому випадку
 $$
 \kappa_1<\kappa_2=|\Delta^{(-\tilde Q)} _{c_1c_2...c_{n-1}c}|+|\Delta^{(-\tilde Q)} _{c_1c_2...c_{n-1}[c+1]}|-|\Delta^{(-\tilde Q)} _{c_1c_2...c_{n-1}c}\cap\Delta^{(-\tilde Q)} _{c_1c_2...c_{n-1}[c+1]}|=W.
 $$
 
 У другому ж випадку $\kappa_2<\kappa_1=W$.
 
Якщо ж циліндри $\Delta^{(-\tilde Q)} _{c_1c_2...c_{n-1}c}$, $\Delta^{(-\tilde Q)} _{c_1c_2...c_{n-1}[c+1]}$ \emph{не  перекриваються} і розташовані: 
 \begin{itemize}
 \item \emph{ ''зліва направо''}, то
 $$
\nu_1=\inf {\Delta^{(-\tilde Q)} _{c_1c_2...c_{n-1}[c+1]}}-\sup {\Delta^{(-\tilde Q)} _{c_1c_2...c_{n-1}c}}=-\kappa_1> 0;
 $$
 \item \emph{ ''справа наліво''}, то
 $$
\nu_2=\inf{\Delta^{(-\tilde Q)} _{c_1c_2...c_{n-1}c}} -\sup {\Delta^{(-\tilde Q)} _{c_1c_2...c_{n-1}[c+1]}}=-\kappa_2>0.
 $$
 \end{itemize}
 
 Проте, в такому разі у першому випадку 
 $$
\nu_1>\nu_2= V=-|\Delta^{(-\tilde Q)} _{c_1c_2...c_{n-1}c}|-|\Delta^{(-\tilde Q)} _{c_1c_2...c_{n-1}[c+1]}|-\varpi,
 $$
 де $\varpi$ --- міра Лебега спільного суміжного з цилідрами $\Delta^{(-\tilde Q)} _{c_1c_2...c_{n-1}c}$, $\Delta^{(-\tilde Q)} _{c_1c_2...c_{n-1}[c+1]}$ інтервала. У другому випадку $V=\nu_1<\nu_2$.

 Перевіримо, які ж із наведених вище співвідношень є справедливими. Отже, нехай $n$ --- парне. Розглянемо різницю 
 $$
\kappa_1\equiv\sup{\Delta^{(-\tilde Q)} _{c_1c_2...c_{n-1}c}}-\inf{\Delta^{(-\tilde Q)} _{c_1c_2...c_{n-1}[c+1]}}=
 $$
 $$
=a_{c,n}\prod^{n-1} _{j=1}{q_{c_j,j}}+q_{c,n}\left(\prod^{n-1} _{j=1}{q_{c_j,j}}\right)\left(a_{0,n+1}+\sum^{\infty} _{t=n+2}{\left[\tilde a_{0,t}\prod^{t-1} _{r=n+1}{\tilde q_{0,r}}\right]}\right)-
 $$
 $$
-a_{c+1,n}\prod^{n-1} _{j=1}{q_{c_j,j}}+q_{c+1,n}\left(\prod^{n-1} _{j=1}{q_{c_j,j}}\right)\left(a_{m_{n+1},n+1}+\sum^{\infty} _{t=n+2}{\left[\tilde a_{m_{t},t}\prod^{t-1} _{r=n+1}{\tilde q_{m_r,r}}\right]}\right)=\left(\prod^{n-1} _{j=1}{q_{c_j,j}}\right)\times
$$
 $$
\times\left(q_{c+1,n}\left(a_{m_{n+1},n+1}+\sum^{\infty} _{t=n+2}{\left[\tilde a_{m_{t},t}\prod^{t-1} _{r=n+1}{\tilde q_{m_r,r}}\right]}\right)+q_{c,n}\left(-1+\sum^{\infty} _{t=n+2}{\left[\tilde a_{0,t}\prod^{t-1} _{r=n+1}{\tilde q_{0,r}}\right]}\right)\right).
$$
 
 Позначивши
 $$
 \omega_1=a_{m_{n+1},n+1}+\sum^{\infty} _{t=n+2}{\left[\tilde a_{m_{t},t}\prod^{t-1} _{r=n+1}{\tilde q_{m_r,r}}\right]},
 $$
 $$
 \omega_2=\sum^{\infty} _{t=n+2}{\left[\tilde a_{0,t}\prod^{t-1} _{r=n+1}{\tilde q_{0,r}}\right]},
 $$
 отримаємо
 $$
 \kappa_1=(q_{c+1,n}\omega_1-q_{c,n}+q_{c,n}\omega_2)q_{c_1,1}q_{c_2,2}...q_{c_{n-1},n-1}.
 $$
 
 Таким чином,  справедливими  є подвійна нерівність
$$
 -q_{c,n}<\frac{\kappa_1}{q_{c_1,1}q_{c_2,2}...q_{c_{n-1},n-1}}\le -q_{c,n}+\max\{q_{c,n},q_{c+1,n}\},
 $$
 та умови
 $$
 \kappa_1<0, ~\text{якщо}~ q_{c+1,n}\omega_1<(1-\omega_2)q_{c,n};
 $$
$$
\kappa_1=0, ~\text{якщо}~ q_{c+1,n}\omega_1=(1-\omega_2)q_{c,n};
$$
 $$
 \kappa_1\ge0, ~\text{якщо}~ q_{c+1,n}\omega_1\ge(1-\omega_2)q_{c,n}.
 $$
 
 Крім того
 $$
 \kappa_2=\sup{\Delta^{(-\tilde Q)} _{c_1c_2...c_{n-1}[c+1]}} -\inf {\Delta^{(-\tilde Q)} _{c_1c_2...c_{n-1}c}}=(q_{c,n}+q_{c+1,n}\omega_2+q_{c,n}\omega_1)\prod^{n-1} _{j=1}{q_{c_j,j}}>0,
 $$
 $$
 \frac{\kappa_2-\kappa_1}{q_{c_1,1}q_{c_2,2}...q_{c_{n-1},n-1}}=(2-\omega_2)q_{c,n}+(q_{c,n}-q_{c+1,n})\omega_1+q_{c+1,n}\omega_2=
$$
$$
=(2-\omega_2)q_{c,n}-(\omega_1-\omega_2)q_{c+1,n}+q_{c,n}\omega_1>0,
 $$
де $\omega_1>\omega_2$.
 
 Отже, у випадку парного $n$  циліндри розташовані ''зліва направо'', але залежно від матриці $\tilde Q$  суміжні циліндри $\Delta^{(-\tilde Q)} _{c_1c_2...c_n}$ можуть або перекриватися, або не перекриватися, або перетинатися в одній точці. 
 
Аналогічно, якщо $n$ --- непарне число, тоді

$$
\kappa_2\equiv\sup{\Delta^{(-\tilde Q)} _{c_1c_2...c_{n-1}[c+1]}}-\inf{\Delta^{(-\tilde Q)} _{c_1c_2...c_{n-1}c}}=-q_{c,n}\prod^{n-1} _{j=1}{q_{c_j,j}}+
$$
$$
+q_{c+1,n}\left(\prod^{n-1} _{j=1}{q_{c_j,j}}\right)\left(a_{m_{n+1},n+1}+\sum^{\infty} _{t=n+2}{\left[\tilde a_{0,t}\prod^{t-1} _{r=n+1}{\tilde q_{0,r}}\right]}\right)+
$$
$$
+q_{c,n}\left(\prod^{n-1} _{j=1}{q_{c_j,j}}\right)\left(\sum^{\infty} _{t=n+2}{\left[\tilde a_{m_{t},t}\prod^{t-1} _{r=n+1}{\tilde q_{m_r,r}}\right]}\right).
$$
Позначивши
$$
\omega_1=a_{m_{n+1},n+1}+\sum^{\infty} _{t=n+2}{\left[\tilde a_{0,t}\prod^{t-1} _{r=n+1}{\tilde q_{0,r}}\right]}
$$
$$
\omega_2=\sum^{\infty} _{t=n+2}{\left[\tilde a_{m_{t},t}\prod^{t-1} _{r=n+1}{\tilde q_{m_r,r}}\right]}
$$
отримаємо
$$
\kappa_2=(-q_{c,n}+q_{c+1,n}\omega_1+q_{c,n}\omega_2)q_{c_1,1}q_{c_2,2}...q_{c_{n-1},n-1}.
$$

Отже, 
$$
-q_{c,n}<\frac{\kappa_2}{q_{c_1,1}q_{c_2,2}...q_{c_{n-1},n-1}}\le-q_{c,n}+\max\{q_{c,n},q_{c+1,n}\}.
$$

Причому,
$$
 \kappa_2<0, ~\text{якщо}~ q_{c+1,n}\omega_1<(1-\omega_2)q_{c,n};
 $$
$$
 \kappa_2=0, ~\text{якщо}~ q_{c+1,n}\omega_1=(1-\omega_2)q_{c,n};
 $$
 $$
 \kappa_2\ge0, ~\text{якщо}~ q_{c+1,n}\omega_1\ge(1-\omega_2)q_{c,n}.
 $$
 
 Оскільки
 $$
 \kappa_1\equiv\sup{\Delta^{(-\tilde Q)} _{c_1c_2...c_{n-1}c}}-\inf{\Delta^{(-\tilde Q)} _{c_1c_2...c_{n-1}[c+1]}}=
$$
$$
=(q_{c,n}+q_{c,n}\omega_1+q_{c+1,n}\omega_2)\left(\prod^{n-1} _{j=1}{q_{c_j,j}}\right)>0
$$
 та
 $$
 \frac{\kappa_1-\kappa_2}{q_{c_1,1}q_{c_2,2}...q_{c_{n-1},n-1}}=(2-\omega_2)q_{c,n}+(q_{c,n}-q_{c+1,n})\omega_1+q_{c+1,n}\omega_2>0,
$$
то у випадку непарного $n$  циліндри розташовані ''справа наліво'', але залежно від матриці $\tilde Q$  суміжні циліндри $\Delta^{(-\tilde Q)} _{c_1c_2...c_n}$ можуть або перекриватися, або не перекриватися, або перетинатися в одній точці. 
\end{enumerate}

Із наведених вище міркувань слідує наступне твердження.

\begin{theorem}
Для довільного числа $x\in[t^{'} _0;t^{''} _0]$ існує послідовність $(i_k)$, $i_k\in N^0 _{m_k}\equiv~\{0,1,...,m_k\}$, така, що
$$
x=-a_{i_1,1}+\sum^{\infty} _{k=2}{\left[(-1)^ka_{i_k,k}\prod^{k-1} _{j=1} {q_{i_j,j}}\right]},
$$
якщо для всіх $k \in \mathbb N$ справедливою є наступна система умов
$$
\left\{
\begin{aligned}
q_{c+1,2k}\left(a_{m_{2k+1},2k+1}+\sum^{\infty} _{t=2}{\left[\tilde a_{m_{2k+t},2k+t}\prod^{2k+t-1} _{r=2k+1}{\tilde q_{m_r,r}}\right]}\right)\ge
q_{c,2k}\left(1-\sum^{\infty} _{t=2}{\left[\tilde a_{0,2k+t}\prod^{2k+t-1} _{r=2k+1}{\tilde q_{0,r}}\right]}\right);\\
q_{c+1,2k-1}\left(a_{m_{2k},2k}+\sum^{\infty} _{t=2}{\left[\tilde a_{0,2k+t-1}\prod^{2k+t-2} _{r=2k}{\tilde q_{0,r}}\right]}\right)\ge
 q_{c,2k-1}\left(1-\sum^{\infty} _{t=2}{\left[\tilde a_{m_{2k+t-1},2k+t-1}\prod^{2k+t-2} _{r=2k}{\tilde q_{m_r,r}}\right]}\right).\\
\end{aligned}
\right.
$$
\end{theorem}

\section{Побудова та дослідження ряду \eqref{def: nega-tilde Q 2}}

Нехай маємо деяку матрицю ${\tilde Q}^{'}=||\tilde q_{i,j}||$ ($i=\overline{0,m_j}$, $m_j \in \mathbb N \cup \{0,\infty\}$, $j=1,2,...$), яка має ті ж самі властивості, що й матриця $\tilde Q$.

Розіб'ємо $[0;1]$ точками $\tilde a_{0,1}, \tilde a_{1,1},..., \tilde a_{m_1,1}$ на відрізки, які, слідуючи ''зліва направо'', позначимо  відповідно $\Delta^{-\tilde Q} _{0}, \Delta^{-\tilde Q} _{1},...,$ і назвемо їх відрізками першого рангу. Очевидно, що $|\Delta^{-\tilde Q} _{i_1}|=\tilde q_{i_1,1}$. Кожен з відрізків першого рангу $\Delta^{-\tilde Q} _{i_1}$, слідуючи  ''справа наліво'', розіб'ємо (точками $\tilde a_{i_1+1,1},\tilde a_{0,2}, \tilde a_{1,2},...$) на відрізки другого рангу $\Delta^{-\tilde Q} _{i_1i}$, де $i=\overline{0,m_2}$, таким чином, що $|\Delta^{-\tilde Q} _{i_1i}|=\tilde q_{i_1,1}\tilde q_{i,2}$ і т. д..  Кожен відрізок $(2k-1)$-го рангу  $\Delta^{-\tilde Q} _{i_1i_2...i_{2k-1}}$,  слідуючи ''справа наліво'', розіб'ємо на відрізки $2k$-го рангу $\Delta^{-\tilde Q} _{i_1i_2...i_{2k-1}i}$, де $i=\overline{0,m_{2k}}$, так, що
$$
|\Delta^{-\tilde Q} _{i_1i_2...i_{2k-1}i}|=\tilde q_{i,2k}\prod^{2k-1} _{j=1} {\tilde q_{i_j,j}},
$$
а кожний відрізок $2k$-го рангу  $\Delta^{-\tilde Q} _{i_1i_2...i_{2k-1}i_{2k}}$, слідуючи  ''зліва направо'', розіб'ємо на відрізки $(2k+1)$-го рангу $\Delta^{-\tilde Q} _{i_1i_2...i_{2k}i}$, $i=\overline{0,m_{2k+1}}$, так, що
$$
|\Delta^{-\tilde Q} _{i_1i_2...i_{2k}i}|=\tilde q_{i,2k+1}\prod^{2k} _{j=1} {\tilde q_{i_j,j}},
$$
і т. д. В результаті отримаємо систему відрізків різних рангів, які володіють властивостями:
\begin{enumerate}
\item $\Delta^{-\tilde Q} _{i_1i_2...i_{k}i}\subset \Delta^{-\tilde Q} _{i_1i_2...i_{k}}$ для всіх $i\in N^{0} _{m_{k+1}}$;
\item
$$
|\Delta^{-\tilde Q} _{i_1i_2...i_{k}}|=\prod^{k} _{j=1}{\tilde q_{i_j,j}}\to 0~~~(k \to\infty).
$$
\end{enumerate}

Тому за аксіомою Кантора кожна послідовність $(i_k)$, $i_k\in N^{0} _{m_k}$, а разом з нею і послідовність відрізків $(\Delta^{-\tilde Q} _{i_1i_2...i_{k}})$ визначає єдину точку $x\in [0;1)$, яку позначатимемо $\Delta^{-\tilde Q} _{i_1i_2...i_{k}...}$. І навпаки, для кожної точки $x\in [0;1)$ існують відрізки $\Delta^{-\tilde Q} _{i_1(x)}$, $\Delta^{-\tilde Q} _{i_1(x)i_2(x)}$, ..., $\Delta^{-\tilde Q} _{i_1(x)i_2(x)...i_k(x)}$, ..., які містять $x$ та 
$$
x=\bigcap^{\infty} _{k=1}{\Delta^{-\tilde Q} _{i_1(x)i_2(x)...i_k(x)}}=\Delta^{-\tilde Q} _{i_1(x)i_2(x)...i_k(x)...}.
$$

Накладемо умови, щоб $\tilde Q^{'}=\tilde Q$. Тобто,
$$
\tilde q_{i_j,j}=\begin{cases}
q_{i_j,j},&\text{якщо $j$ --- непарне;}\\
q_{m_j-i_j,j},&\text{якщо $j$ --- парне.}
\end{cases}
$$

Рівність $x=\Delta^{-\tilde Q} _{i_1(x)i_2(x)...i_k(x)...}$ називатимемо \emph{нега-$\tilde Q$-зображенням числа $x$}.

Якщо $x=\Delta^{-\tilde Q} _{i_1(x)i_2(x)...i_k(x)...}$, то існує рівно $i_1$ відрізків 1-го рангу, які лежать лівіше точки $x$; 
$m_2-i_2$ відрізків 2-го рангу, які належать $\Delta^{-\tilde Q} _{i_1(x)}$ і лежать лівіше $x$; $i_{2k-1}$ відрізків $(2k-1)$-го рангу, які належать $\Delta^{-\tilde Q} _{i_1(x)i_2(x)...i_{2k-2}(x)}$, лежать лівіше $x$ і мають сумарну довжину 
$$
a_{i_{2k-1},2k-1}\prod^{2k-2} _{j=1}{\tilde q_{i_j,j}};
$$
$m_{2k}-i_{2k}$ відрізків $2k$-го рангу, які належать $\Delta^{-\tilde Q} _{i_1(x)i_2(x)...i_{2k-1}(x)}$, лежать лівіше  $x$ і мають сумарну довжину
$$
a_{m_{2k}-i_{2k},2k}\prod^{2k-1} _{j=1}{\tilde q_{i_j,j}},
$$
і т. д. Тому для будь-якого $x\in [0;1)$ існує послідовність $(i_k(x))$,  $i_k(x)\in N^{0} _{m_k}$, така, що 
\begin{equation}
\label{def: nega-tilde Q 3}
x=a_{i_1(x),1}+\sum^{\infty} _{k=2}{\left[\tilde a_{i_k(x),k}\prod^{k-1} _{j=1}{\tilde q_{i_j(x),j}}\right]}.
\end{equation}

\begin{remark} Варто зазначити, що розклад \eqref{def: nega-tilde Q 3} є узагальненням взаємозв'язку згаданих вище знакопочережних представлень із відповідними знакододатними представленнями. Тобто, 
\begin{itemize}
\item якщо $q_{i_j,j}=\frac{1}{s}$, $1<s$ --- фіксоване натуральне число, $i=\overline{0,s-1}$, $j=1,2,...$, тоді ряд \eqref{def: nega-tilde Q 3}  набуде виду
$$
\frac{\alpha_1}{s}+\frac{s-1-\alpha_2}{s^2}+...+\frac{\alpha_{2k-1}}{s^{2k-1}}+\frac{s-1-\alpha_{2k}}{s^{2k}}+...\equiv \frac{1}{s+1}-\sum^{\infty} _{k=1}{\frac{(-1)^k\alpha_k}{s^k}};
$$
\item якщо $q_{i_j,j}=\frac{1}{d_j}$, $(d_j)$ --- фіксована послідовність натуральних чисел, більших~$1$, $i=\overline{0,d_j-1}$, $j=1,2,...$, тоді ряд \eqref{def: nega-tilde Q 3}  набуде виду
$$
\frac{\varepsilon_1}{d_1}+\frac{d_2-1-\varepsilon_2}{d_1d_2}+\frac{\varepsilon_3}{d_1d_2d_3}+\frac{d_4-1-\varepsilon_4}{d_1d_2d_3d_4}+...\equiv \sum^{\infty} _{k=1}{\frac{d_{2k}-1}{d_1d_2...d_{2k}}}-\sum^{\infty} _{k=1}{\frac{(-1)^k\varepsilon_k}{d_1d_2...d_{k}}}.
$$
\end{itemize}
\end{remark}

Таким чином,
$$
\Delta^{-\tilde Q} _{i_1i_2...i_k...}\equiv \Delta^{\tilde Q} _{i_1[m_2-i_2]...i_{2k-1}[m_{2k}-i_{2k}]...}
$$
і навпаки
$$
 \Delta^{\tilde Q} _{i_1i_2...i_k...}\equiv\Delta^{-\tilde Q} _{i_1[m_2-i_2]...i_{2k-1}[m_{2k}-i_{2k}]...}.
$$

Оскільки
$$
a_{{m_{2k}-i_{2k}},2k}=q_{0,2k}+q_{1,2k}+...+q_{m_{2k}-i_{2k}-1,2k}=1-q_{m_{2k}-i_{2k},2k}-q_{m_{2k}-i_{2k}+1,2k}-...-q_{m_{2k},2k},
$$
то
$$
\Delta^{-\tilde Q} _{i_1i_2...i_k...}\equiv \sum^{i_1-1} _{i=0} {q_{i,1}}+\sum^{\infty} _{k=2}{\left[(-1)^{k-1}\tilde \delta_{i_k,k}\prod^{k-1} _{j=1}{\tilde q_{i_j,j}}\right]}+\sum^{\infty} _{k=1}{\left(\prod^{2k-1} _{j=1}{\tilde q_{i_j,j}}\right)}.
$$

\begin{center}

\emph{Ключові слова:} дійсне число, нега-s-кове зображення, ряди Кантора, $\tilde Q$-зображення,  нега-$\tilde Q$-зображення.

\emph{Keywords:} real number, nega-s-adic representation, Cantor series, $\tilde Q$-representation, nega-$\tilde Q$-representation.

\end{center}

\end{document}